\documentclass[a4paper,12pt]{amsart}
\usepackage{version}
\usepackage{amssymb}
\usepackage{amsfonts}
\usepackage{graphicx}
\usepackage{tikz}
\usetikzlibrary{arrows}

\usepackage{xcolor}
\usepackage[pagebackref]{hyperref}
\hypersetup{
   colorlinks,
    linkcolor={red!60!black},
    citecolor={blue!60!black},
    urlcolor={blue!90!black}
}

\usepackage{epstopdf}
\usepackage[english]{babel}
\usepackage{float}
\usepackage{cite}
\usepackage{tikz,pgf}
\usetikzlibrary{patterns,spy}

\usepackage{geometry}
\usepackage[pagebackref]{hyperref}
\hypersetup{
    colorlinks,
    linkcolor={red!60!black},
    citecolor={blue!60!black},
   urlcolor={blue!90!black}
}

\newtheorem{theorem}{Theorem}[section]

\newtheorem{lemma}[theorem]{Lemma}
\newtheorem{corollary}[theorem]{Corollary}
\newtheorem{proposition}[theorem]{Proposition}

\theoremstyle{definition}
\newtheorem{definition}[theorem]{Definition}

\theoremstyle{remark}
\newtheorem{remark}[theorem]{Remark}

\numberwithin{equation}{section}

\renewcommand\bigskip{\medskip}

\def\to{\rightarrow}

\def\cF{\mathcal{F}}
\def\N{\mathbb N}

\def\cM{\mathcal{M}}

\def\R{\mathbb R}

\def\Z{\mathbb Z}

\newcommand{\iii}{\mathtt{i}}
\newcommand{\jjj}{\mathtt{j}}
\newcommand{\kkk}{\mathtt{k}}

\DeclareMathOperator{\diam}{diam}
\DeclareMathOperator{\spt}{spt}
\DeclareMathOperator{\dimh}{\dim_H}

\begin{document}

\title[]{Covering the Sierpi{\'n}ski carpet with tubes}

\author{Aleksi Py\"or\"al\"a}
\email{aleksi.pyorala@oulu.fi}
\address{Research Unit of Mathematical Sciences,
        P.O.Box 8000, FI-90014,  University of Oulu,
    Finland}
\author{Pablo Shmerkin}
\email{pshmerkin@utdt.edu}
    \address{Department of Mathematics and Statistics, Torcuato Di Tella University, and CONICET, Buenos Aires, Argentina}
\author{Ville Suomala}
\email{ville.suomala@oulu.fi}
\address{Research Unit of Mathematical Sciences,
        P.O.Box 8000, FI-90014,  University of Oulu,
    Finland}
\author{Meng Wu}
\email{meng.wu@oulu.fi}
\address{Research Unit of Mathematical Sciences,
        P.O.Box 8000, FI-90014,  University of Oulu,
    Finland}

\thanks{The work of A.P., V.S., and M.W. has been supported by the Academy of Finland.
A.P. acknowledges the support of the University of Oulu Graduate School. PS has received funding from the European Research Council (ERC) under the European Union's Horizon 2020 research and innovation programme (grant agreement No. 803711)}


\begin{abstract}
We show that non-trivial $\times N$-invariant sets in $[0,1]^d$, such as the Sierpi\'{n}ski carpet and the Sierpi\'{n}ski sponge, are tube-null, that is, they can be covered by a union of tubular neighbourhoods of lines of arbitrarily small total volume. This introduces a new class of tube-null sets of dimension strictly between $d-1$ and $d$. We utilize ergodic-theoretic methods to decompose the set into finitely many parts, each of which projects onto a set of Hausdorff dimension less than $1$ in some direction. We also discuss coverings by tubes for other self-similar sets, and present various applications.
\end{abstract}

\maketitle

\section{Introduction}

We call a \textit{tube} $T$ of width $w = w(T) > 0$ a $w$-neighborhood of some line in $\R^d$, where from now on $d$ is some integer $\ge 2$. A set $K\subset \R^d$ is called \textit{tube-null} if for every $\varepsilon > 0$ there exists a countable family of tubes $\lbrace T_i \rbrace$ such that $K \subset \bigcup_i T_i$ and $\sum_i w(T_i)^{d-1} < \varepsilon.$

The notion of tube-nullity has its roots in harmonic analysis. It was shown by Carbery, Soria, and Vargas \cite[Theorem 4]{CSV} that if $K$ is a tube-null subset of the unit ball $B(0,1)\subset \R^d$, then there exists a function $f \in L^2(\R^d)$ which is identically zero on $B(0,1)$ and for which the Fourier localisations
$$
S_R f(x) = \int_{|\xi| < R} \widehat{f}(x) e^{2\pi i \xi \cdot x}\, d\xi
$$
fail to converge as $R \to \infty$ for every $x \in K$. It is an open problem to characterize all such \emph{divergence sets} for $S_R$; in particular, it is not known if each such set is tube-null. Note that if the assumption $\spt f\subset\R^d\setminus B(0,1)$ is dropped, then it is not even known if the divergence set is Lebesgue null.

The notion of tube-nullity is also very natural from the point of view of geometric measure theory and, along with several variants, it has been considered in many works, see e.g.  \cite{CW, Carbery, Harangi, Orponen, SS, Chen, SS2}. See also \S \ref{ss:rot} for a variant called tube-dimension. Despite the growing literature on tube-null sets, it is often difficult to verify whether a given set is tube-null or not.  Often, but certainly not always,  the connection between tube-nullity and geometric measure theory arises from orthogonal projections: If a set $K\subset\R^d$ may be decomposed into countably many subsets each of which projects onto a Lebesgue null set under some orthogonal projection $P\colon\R^d\to\R^{d-1}$, then it is easy to see that $K$ is tube-null. On the other hand, if $K$ supports a non zero measure $\mu$ such that its orthogonal projections are all absolutely continuous with a uniformly bounded density, then $\mu(T)\lesssim w(T)$ for all tubes, and a simple computation shows that $K$ is not tube-null.
%
%CHECK THE REFERENCES
%

Since orthogonal projections cannot increase Hausdorff dimension, it is obvious that sets with Hausdorff dimension $<d-1$ are tube-null. Using the Besicovitch-Federer projection theorem, Carbery, Soria and Vargas \cite[Proposition 8]{CSV}  have shown that in $\R^d$, sets with $\sigma$-finite $(d-1)$-dimensional Hausdorff measure are tube-null. Given these facts, the question about tube-nullity is interesting for sets of Hausdorff dimension at least $d-1$. Using a random construction, a variant of the fractal percolation, Shmerkin and Suomala \cite{SS} showed that there are \emph{non} tube-null sets of any dimension $s\in[d-1,d]$; for $s\in ]d-1,d]$ they can even be taken to be Ahlfors-regular. Carbery, Soria, and Vargas had shown this before for the values $s\in]3/2,2]$ by investigating rotationally invariant Cantor sets \cite[Proposition 6]{CSV}. These Cantor targets also provide an interesting example when the dimension is in $]1,3/2[$: They are tube-null, but, as far as we know, no proof using orthogonal projections is available.

Note that tube-nullity itself does not impose any bounds on the dimension of the set: A set of full Hausdorff dimension may easily have a Lebesgue null projection and hence be tube-null. For instance, consider a Cantor set $C\subset[0,1]^{d-1}$ with $\dimh C=d-1$ and $\mathcal{H}^{d-1}(C)=0$ and let $K=C\times [0,1]$. (Here and below, $\dimh$ denotes Hausdorff dimension.) Nevertheless, heuristically it seems reasonable that (absent any special structure as above) sets of larger dimension may have more difficulty being tube-null.

In this work, we investigate the problem of tube-nullity for self-similar sets. Besides the obvious situation in which one of the orthogonal projections onto a hyperplane has $(d-1)$-measure zero, not much is known. A remarkable exception is the von Koch snowflake curve, that was shown to be tube-null by Harangi \cite{Harangi}. In fact, using combinatorial and probabilistic arguments, he showed that the Koch curve may be decomposed into three pieces, each of which projects onto a set of dimension $<1$ in one of the natural directions that appear in the finite level approximation of the fractal curve.

In this paper we extend the class of known tube-null sets by all $\times N$-invariant sets of dimension less than $d$. Included among these are the ``$N$-adic'' self-similar sets, such as the classical Sierpi{\'n}ski carpet in $\R^2$ and the Sierpi{\'n}ski sponge in $\R^3$. See Figure \ref{carpets}. The fact that these sets are tube-null might be surprising, since they are in some sense highly connected, and their dimension is close to maximal.

\begin{figure}[H]
	\includegraphics[width=0.85\textwidth]{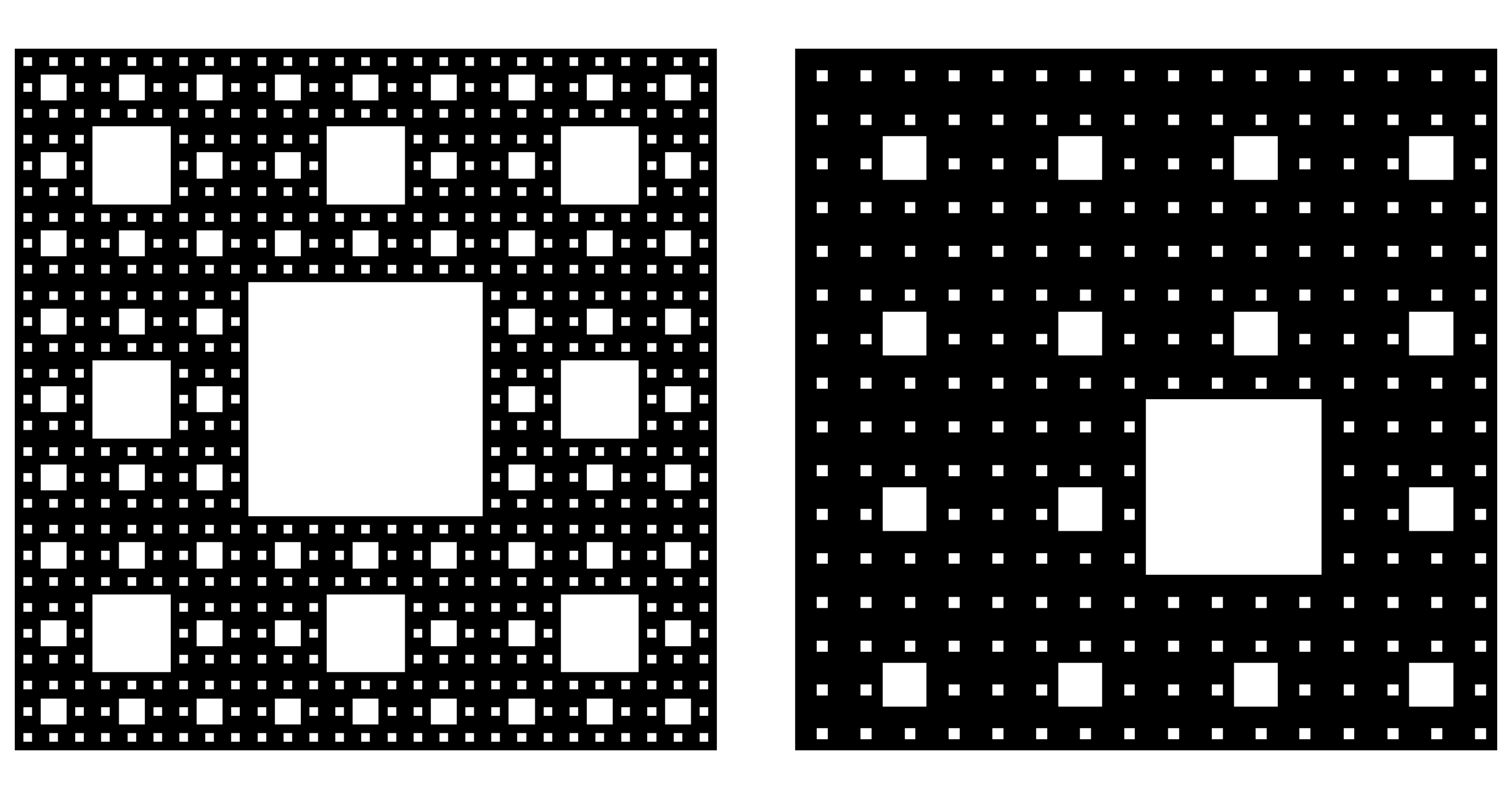}
	\caption{The Sierpi{\'n}ski carpet (left) and a $\times 4$-invariant carpet of dimension $\log 15/\log 4\approx 1.953$ (right). Our main result implies that these sets are tube-null, in a strong sense.}
	\label{carpets}
\end{figure}

Shmerkin and Suomala \cite{SS} showed that stochastically self-similar fractals, like the fractal percolation limit sets, of dimension strictly between $d-1$ and $d$ are \emph{not} tube-null. Since the fractal percolation process on the $N$-adic grid is invariant under $\times N$ in distribution, this highlights a difference between how deterministically and stochastically  self-similar sets differ in their tube covering behaviour.

To formulate our main result, let $T_N$ denote the $\times N$ map on the $d$-dimensional torus identified with $[0,1]^d$, that is, $(x_1,\ldots,x_d)\mapsto (Nx_1 \mod 1,\ldots,Nx_d \mod 1)$. Note that $T_N$ depends on $d$, although we do not make this dependence explicit. In our terminology, a direction is a non-zero vector in $\R^d$, and the projection in direction $v$ is the orthogonal projection onto the line spanned by $v$.
\begin{theorem}\label{thm-principal-1}
Let $K\subsetneq [0,1]^d$ be a closed $T_N$-invariant set. Then, there is $c<1$ and a finite collection $\mathcal{V}$ of directions such that
\[K=\bigcup_{v\in\mathcal{V}}E_v\,,\]
where the projection of $E_v$ in direction $v$ has Hausdorff dimension $<c$.
\end{theorem}

\begin{corollary} \label{cor:main}
Let $K\subsetneq [0,1]^d$ be a closed, $T_N$-invariant set. Then $K$ is tube-null. Moreover, there is $c<1$ such that for every $\varepsilon>0$ there are hyperplanes $(H_i)$ and numbers $r_i$ with $\sum_i r_i^{c}<\varepsilon$, such that $K\subset \bigcup_i H_i(r_i)$. Here $H(r)$ denotes the $r$-neighborhood of $H$.
\end{corollary}

To deduce the corollary from Theorem \ref{thm-principal-1}, for each $v\in\mathcal{V}$, we cover the angle $v$-projection of $E_v$ by intervals $B(x_i,r_i)$ with $\sum_i r_i^c$ small, pull-them back to obtain neighborhoods of hyperplanes, and take the union over all $v$. The claim that $K$ is tube-null follows by covering $H_i(r_i)\cap[0,1]^d$ by $\sim r_i^{2-d}$ tubes of width $r_i$.

Let us  briefly discuss the quantitative aspect of these results. Both the set of directions $\mathcal{V}$ and the number $c<1$ provided by Theorem \ref{thm-principal-1} depend on $K$. More precisely, they depend only on $d, N$, and the number
\[\alpha(K)=\sup\{t>0\,:\,\exists B(y,t)\subset[0,1]^d\setminus K\}\,.\]
In particular, we can provide an explicit collection $\mathcal{V}$ depending only on $d,N$, and $\alpha(K)$ that satisfies the claim of Theorem \ref{thm-principal-1}, see Section \ref{sec:R_0}.
The existence of the constant $c$ is deduced via a compactness argument and we don't have an explicit bound on it in addition to $c<1$. Moreover, as will be seen, the definition of the sets $E_v$ is rather abstract in terms of typical orbits of certain invariant measures. In particular, we are not able to provide an explicit definition for $E_v$. It is obvious that they cannot be open in the relative topology of $K$ and by the Baire category theorem, not all of them can be closed.

We now sketch the main idea of the proof of Theorem \ref{thm-principal-1}.
We first observe, by investigating the Fourier coefficients of the invariant measures and their projections, that there exists a finite collection $\mathcal{V}$ of rational directions such that for any $T_N$-invariant measure on $K$, the orthogonal projection of the measure in at least one of these directions is not absolutely continuous. We proceed by showing that for the projected measure, non-absolute continuity actually implies that the entropy dimension is $<c<1$. This is shown in Proposition \ref{prop max dim implies abs cont}, using the weak-separation condition for the projected iterated function system. A compactness argument then implies that $c$ can be taken to be uniform over all invariant measures.
We now define the sets $E_v$, using the aforementioned directions $v\in\mathcal{V}$.
%We note that in the symbolic coding of $K$, for each $x\in K$ there is an invariant measure $\mu$ such that $x$ is generic with respect to its Birkhoff averages.
We note that the orbit of each $x\in K$ under $T_N$ equidistributes along some subsequence for (at least) some $T_N$-invariant measure $\mu$.
If the projection of $\mu$ in the direction $v$ satisfies $\dim\mu<c$, we include $x$ in $E_v$. It is then immediate that $\lbrace E_v\rbrace_{v\in\mathcal{V}}$ is a cover for $K$.
Finally, to show that the projection of each $E_v$ in the direction $v$ has Hausdorff dimension $<c$,
we use Bowen's lemma, a form of the variational principle: an upper bound for the entropies of invariant measures on a set also serves as an upper bound for the topological entropy of the set. See Lemma \ref{lemma-Bowen,1}.

While Theorem \ref{thm-principal-1} applies to many self-similar sets,  $\times N$-invariance is crucial in the proof. We are able to prove that many homogeneous self-similar sets with no rotations are also tube-null.
\begin{proposition}\label{prop: homo-small}
Let $K\subset\R^d$ be the attractor of an iterated function system $\{f_i(x)=rx+\lambda_i\}_{i=0}^{m-1}$. If $-\log_2r>\log_2 m-\frac{2}{m}$, then there is  a finite collection $\mathcal{V}$ of directions  and a cover $(E_v)_{v\in\mathcal{V}}$ of $K$, such that the projection of $E_v$ in direction $v$ has Hausdorff dimension $\le \frac{\log_2 m-2/m}{-\log_2 r }<1$. In particular, $K$ is tube-null.
\end{proposition}
This proposition gives many new examples of tube-null self-similar sets with dimension $>1$ and without Lebesgue null projections. Note that we are not assuming any separation condition for the pieces $f_i(K)$, nor are we requiring a grid structure. See Figure \ref{sss} in \S\ref{ss:rot} below for a non-trivial example of a self-similar set to which Proposition \ref{prop: homo-small} applies. In \S\ref{ss:rot} we show that, on the other hand, planar self-similar sets of dimension $>1$ such that the generating IFS spans an infinite rotation group do not admit such a partition $E_v$. While stopping short of proving that they are not tube-null, we show that their \emph{tube dimension} is $1$, see \S\ref{ss:rot} for details.

The meta structure of the proof of Proposition \ref{prop: homo-small} is similar to that of Theorem \ref{thm-principal-1}, but the details are much simpler.
For instance, the set of directions $\mathcal{V}$ is given by the directions of exact overlap: For each pair $0\le i<j\le m-1$, there is a direction $v$ such that the pieces $f_i(K)$ and $f_j(K)$ have identical projection onto this direction. The sets $E_v$ are defined via digit frequencies in the symbolic space, and this allows us to conclude the quantitative bound $\frac{\log_2 m-2/m}{-\log_2 r }$ for the dimension. We include a  proof for Proposition \ref{prop: homo-small} in the last section along with the final remarks. Despite the similarities to  Theorem \ref{thm-principal-1}, the proof is self-contained and, it could serve as a good warm up for the more involved arguments in Sections \ref{sec:ws}--\ref{sec:proof}.

The paper is organized as follows. In Section \ref{sec_prel}, we set up the notation and recall the necessary tools from ergodic theory and fractal geometry. In Section \ref{sec:ws}, we proceed by showing that shift invariant measures on homogeneous affine iterated function systems on the real-line are either absolutely continuous or have dimension $<1$. This observation is crucial for the proof of Theorem \ref{thm-principal-1}, which is the content of Section \ref{sec:proof}. In the Section \ref{ss:four}, we apply the results of Section \ref{sec:ws} to conclude that for each invariant measure, some rational projection has dimension $<c$. We note that this remains true also for measures invariant under the iterates $T_N^m$ and in Section \ref{ss:hl}, we build such a high-level iterated function system to effectively estimate the dimension of the projections of invariant measures via entropy estimates. In Section \ref{sec:final_proof}, the proof of Theorem is completed by defining the covers $(E_v)_v$ and estimating the dimension of their projections using Bowen's lemma.
In the final Section \ref{sec:rems}, we provide a quantitative bound on $\mathcal{V}$, provide a proof for Proposition \ref{prop: homo-small} and discuss a few additional results and applications related to tube-dimension and isotropic doubling measures.

\section{Preliminaries}\label{sec_prel}

%\subsection{The shift space}

In this section we review some useful concepts from fractal geometry and ergodic theory, setting up notation along the way.
Let $\Gamma$ be a finite set of cardinality $\#\Gamma\ge 2$. A family of contracting functions $\cF = \lbrace f_i \rbrace_{i \in \Gamma}$ on a complete metric space is called an \emph{iterated function system} (IFS). It is well-known that for an IFS $\cF$ there exists a unique compact set $K$, called the \emph{attractor} of $\cF$, satisfying
$$
K = \bigcup_{i \in \Gamma} f_i(K).
$$
%Without loss of generality, we always assume that $\text{diam}\,K = 1$.
If the contractions $f_i$ are similitudes, i.e. $|f_i(x) - f_i(y)| = r_i |x-y|$ for some $r_i\in (0,1)$ and all $x, y$, we call $\cF$ and its attractor $K$ self-similar. If $r_i \equiv r$, we call the IFS homogeneous. For an $n \geq 2$ and $\iii \in \Gamma^n$, we use the notations $f_\iii = f_{i_1} \circ \cdots \circ f_{i_n}$ and $K_\iii = f_\iii(K)$. %\cite{Falconer}.

For a word  $\iii = (i_1, i_2, \ldots)$, $i_j\in\Gamma$, of length $\ge k$ (possibly infinite), we denote  $\iii|_k = (i_1, \ldots, i_k)$. For $\iii \in \Gamma^n$, we define the cylinder $[\iii]$ as the set of all infinite words $\jjj \in \Gamma^\N$ for which $\jjj|_n = \iii$. By $\pi$ we denote the natural coding of $K$ by $\Gamma^\N$,
$$
\pi(\iii) = \bigcap_{k=1}^\infty K_{\iii|_k} = \lim_{k\to\infty} f_{i_1}\cdots f_{i_k}(0).
$$
%We extend the notation $\pi(\iii)$ to finite words in the obvious way. \textbf{ARE WE USING THIS FOR FINITE WORDS? IF NOT, SHOULD BE REMOVED...}

Let $\sigma\colon\Gamma^\N\to\Gamma^\N$, $\sigma(i_1, i_2, \ldots ) =(i_2, i_3, \ldots)$ denote the left shift. The \emph{symbolic space} $\Gamma^\N$ is always equipped with the topology generated by the cylinder sets. By a measure on $\Gamma^\N$ we always mean a finite Borel measure.

The push-forward of a measure $\mu$ through a measurable function $f$ is denoted by $f\mu = \mu \circ f^{-1}$. Given a compact topological space $X$ and a continuous transformation $S: X \to X$, the pair $(X,S)$ is called a \textit{dynamical system}. A measure $\mu$ on $X$ is called $S$\textit{-invariant} if $S\mu = \mu$; we let $\cM_{\rm inv}(X,S)$ denote the set of $S$-invariant probability measures on $X$. %Let $\cM_{\rm erg}(X,S)\subset\cM_{\rm inv}(X,S)$ denote the set of ergodic measures.
On the symbolic space $X = \Gamma^\N$,  the space of probability measures is (weak-$^*$) compact and thus, as a closed subset of a compact space, the set of invariant measures is also compact.

Throughout the paper $\log$ denotes logarithm to base $2$. Given a measure space $(X, \mu)$ and a measurable finite partition $\mathcal{A}$ of $X$, the \emph{Shannon entropy} of $\mu$ with respect to the partition $\mathcal{A}$ is defined
$$
H(\mu, \mathcal{A}) = -\sum_{A \in \mathcal{A}} \mu(A) \log \mu(A).
$$
%It is well known (and easy to see) that $\mu \mapsto H(\mu, \mathcal{A})$ is concave and increases if $\mu$ is fixed and the partition refined, and that $H(\mu,\mathcal{A})\le \log\#\mathcal{A}$. Even though $H(\mu, \mathcal{A})$ is not convex in $\mu$, it does hold that
We will make use of the following elementary property of $H$: If $p=(p_i)_{i=1}^M$ is a probability vector, then
\begin{equation} \label{eq-entropy-convexity}
H\left(\sum_{i=1}^M p_i \mu_i,\mathcal{A}\right) \le H(p)+\sum_{i=1}^M p_i  H(\mu_i,\mathcal{A}) \le \log M+\sum_{i=1}^M p_i  H(\mu_i,\mathcal{A}),
\end{equation}
where $H(p)=\sum_{i=1}^M p_i\log(1/p_i)$.

If another partition $\mathcal{B}$ is given, the conditional entropy is defined as
\[
H(\mu, \mathcal{A}|\mathcal{B}) = \sum_{B\in\mathcal{B}:\mu(B)>0} \mu(B) H(\mu_B,\mathcal{A}),
\]
where $\mu_B=\tfrac{1}{\mu(B)}\mu|_B$. Then it holds that
\begin{equation} \label{eq-entropy-refinement}
H(\mu,\mathcal{A}\vee \mathcal{B}) = H(\mu,\mathcal{A})+H(\mu,\mathcal{B}|\mathcal{A}) \le H(\mu,\mathcal{A})+H(\mu,\mathcal{B}).
\end{equation}
where $\mathcal{A}\vee\mathcal{B}$ is the common refinement of the partitions $\mathcal{A}$ and $\mathcal{B}$.
%, that is, it consists of all the intersections of an element of $\mathcal{A}$ and an element of $\mathcal{B}$.
A consequence of this is the following continuity property of $H$ with respect to the partition. % (see \cite[Lemma 3.2]{Hochman}).
\begin{lemma}\label{lemma almost continuity of entropy}
Let $(X, \mu)$ be a measure space, and let $\mathcal{A} , \mathcal{B}$ be  two
finite measurable partitions of $X$. Suppose that for some $M\in\N$ it holds that each element
of $\mathcal{B}$ intersects at most $M$ elements of $\mathcal{A}$, and vice versa. Then
$$
|H(\mu,\mathcal{A})-H(\mu,\mathcal{B})|\le \log M.
$$
\end{lemma}
%
%
%
%
%
\begin{comment}%We are using the measure theoretic entropy only in the shift space...
If $\mu$ is an invariant probability measure on a dynamical system $(X,S)$, and $\mathcal{Q}$ is a finite (measurable) partition of $X$, the entropy of $\mu$ with respect to the partition $\mathcal{Q}$ is defined as
$$
h(\mu, S, \mathcal{Q}) = - \lim_{n \to \infty} \frac1n \sum_{Q \in \bigvee_{k=1}^{n}S^{-k}\mathcal{Q}} \mu(Q) \log \mu(Q),
$$
where $\mathcal{Q}_1 \vee \mathcal{Q}_2$ denotes the refinement of $\mathcal{Q}_1$ and $\mathcal{Q}_2$, i.e. the family consisting of all intersections of the sets in $\mathcal{Q}_1$ and $\mathcal{Q}_2$. A subadditivity argument using \eqref{eq-entropy-refinement} shows the existence of the limit above. The \textit{(measure theoretic) entropy} of $\mu$ is then defined
$$
h(\mu, S) = \sup_{\mathcal{Q}} h(\mu, S, \mathcal{Q}),
$$
where the supremum is taken over all finite partitions of $X$. The Kolmogorov-Sinai theorem states that the supremum is achieved through a generating partition of $(X, S)$; in the case $X= \Gamma^\N$ and $S= \sigma$,
\end{comment}

If $\mu$ is a measure on the real line and $D_n(\R)$ denotes the partition of $\R$ into dyadic intervals of length $2^{-n}$, we use the notation
$$
H_n(\mu) = H(\mu, D_n(\R))
$$
and refer to this number as the $n$-scale entropy of $\mu$. The following quantity, often called the \emph{(lower) entropy dimension} of $\mu$, is what we adapt as our definition of dimension for measures:
$$
\dim \mu = \liminf_{n \to \infty} \frac{H_n(\mu)}{n}.
$$

On $(\Gamma^\N,\sigma)$, we also consider the \emph{(measure theoretic) entropy} for invariant  probability measures, defined by
$$
h(\mu, \sigma) = -\lim_{n \to \infty} \dfrac{1}{n}\sum_{\iii \in \Gamma^n} \mu[\iii] \log \mu[\iii]=\inf_{n\in\N}\dfrac{1}{n}\sum_{\iii \in \Gamma^n} -\mu[\iii] \log \mu[\iii]\,.
$$

We now recall Bowen's definition of topological entropy in the context of the symbolic space $(\Gamma^\N,\sigma)$.
%for non-compact, non-invariant set
%\cite{Bowen}. Let $(X,S)$ be a dynamical system.
Fix an open cover $\mathcal{A}$ of $\Gamma$. Given $B\subset \Gamma^\N$, we define
\[
n_{\mathcal{A}}(B) = \sup\{ n: \text{ for all } k=1,\ldots, n, \text{ there is } A_k\in\mathcal{A} \text{ containing } \sigma^k(B)\}\,,
\]
where we allow $n=+\infty$ and the supremum of the empty set is considered to be $0$. We set $\diam_{\mathcal{A}}(B) = 2^{-n_{\mathcal{A}}(B)}$. (We use exponential to base $2$ in order to match the convention that our logarithms are always to base $2$.) We now define $h_{\textrm{top}}(E,\mathcal{A})$ by following the definition of Hausdorff dimension, using $\diam_{\mathcal{A}}$ instead of the metric diameter: set
\begin{align*}
m_{\mathcal{A}}^s(E) &= \lim_{\varepsilon\to 0} \inf \left\{ \sum_i \diam_{\mathcal{A}}(B_i)^s: E\subset \bigcup_i B_i, \diam_{\mathcal{A}}(B_i)<\varepsilon \right\},\\
h_{\textrm{top}}(E,\mathcal{A}) &= \inf\{s:  m_{\mathcal{A}}^s(E)=0 \}.
\end{align*}
Finally, the \emph{topological entropy} $h_{\rm{top}}(E,\sigma)$ is the supremum of $h_{\textrm{top}}(E,\mathcal{A})$ over all finite open covers $\mathcal{A}$ of $\Gamma^\N$.

It follows easily from the definitions that if $\cF$ is a homogeneous IFS with contraction ratio $r$ associated to the symbolic space $(\Gamma^\N, \sigma)$ and  $E \subset \Gamma^\N$, %and $\pi$ denotes the natural projection of $\Gamma^\N$ onto the attractor of $\cF$,
then
\begin{equation} \label{eq:dimension-from-entropy}
\dimh(\pi E) \leq \dfrac{h_{\rm top}(E, \sigma)}{-\log r}.
\end{equation}
%
%
\begin{comment}
Indeed, consider the metric on $\Gamma^\N$ given by
\[
d(\iii,\jjj) = r^{-N(\iii,\jjj)},
\]
where $N(\iii,\jjj)$ is the length of the common beginning of $\iii$ and $\jjj$ (possibly $0$ or $+\infty$). Notice that balls are cylinder sets. If we let $\mathcal{A}$ be the cover of $\Gamma^\N$ by cylinders of length $1$, it is easy to see that
\[
\diam_d(B) = \diam_{\mathcal{A}}(B)^{-\log r},
\]
and hence
\[
\dimh(E)\le h_{\rm top}(E, \sigma,\mathcal{A})/\log (1/r)\le h_{\rm top}(E,\sigma)/\log(1/r),
\]
where $\dim_{\rm{H}}(E)$ is Hausdorff dimension with respect to the metric $d$. Since $\pi$ is easily seen to be Lipschitz in this metric, the claim \eqref{eq:dimension-from-entropy} follows.
\end{comment}
%
%
%

The variational principle in the form of the following Bowen's lemma is crucial for us.
For $\iii \in \Gamma^\N$, we denote by $V(\iii)$ the collection of the weak$^*$ accumulation points of the sequence
$
\frac{1}{n}\sum_{k=0}^{n-1}\delta_{\sigma^{k}(\iii)}
$,
where $\delta_\jjj$ denotes the Dirac unit mass located at $\jjj\in\Gamma^\N$.

\begin{lemma}[{\cite[Theorem 2]{Bowen}}]  \label{lemma-Bowen,1}
	Consider a symbolic space $(\Gamma^\N,\sigma)$ and let $t\ge 0$. Then the topological entropy of the set
	$$
	\left\{\iii\in \Gamma^\N: \exists \mu\in V(\iii) \textrm{ such that } h(\mu,\sigma)\le t \right\}
	$$
	is at most $t$.
\end{lemma}

For invariant measures on a symbolic space associated to an IFS in which the amount of overlaps is controlled, the dimension of the natural projection is closely connected to the entropy of the measure in the symbolic space.
\begin{lemma}\label{entropy dimension}\label{lemma-8}
Suppose that $\mathcal{G}=\{g_i(x)=rx+t_i\}_{i\in \Gamma}$ is a homogeneous IFS on $\R$ with attractor $E$, and that there exists an $M\in \N$ such that for each $i\in \Gamma$,
$$\left|\{j\in \Gamma:g_i(\emph{ch}(E))\cap g_j(\emph{ch}(E))\neq \emptyset\}\right|\le M,$$
where $\emph{ch}(E)$ denotes the convex hull of $E$. Then for each $\nu\in \cM_{\rm inv}(\Gamma^\N,\sigma)$, if $\pi$ denotes the natural projection from $\Gamma^\N$ to $E$, we have
$$\dim \pi\nu \ge \frac{h(\nu,\sigma)}{-\log r}-\frac{\log M}{-\log r}.
$$
\end{lemma}

\begin{remark}
Note that in the situation when there are no overlaps between the cylinders $g_i(E)$'s (for instance, if $E$ satisfies the strong separation condition), we have  $\dim \pi\nu= \frac{h(\nu,\sigma)}{-\log r}$.
\end{remark}
\begin{proof}
Without loss of generality, we may assume that $\diam(E)=1$.  If $h = h_n = \lfloor -n(\log r)^{-1} \rfloor$ is the integer for which $r^{h+1} < 2^{-n} \le r^h$, then for any $Q\in D_n(\R)$, we have
$$
\left| \lbrace \jjj \in \Gamma^h:\ g_\jjj(\text{ch}(E)) \cap Q \neq \emptyset \rbrace \right| \leq 2 M^h.
$$
Indeed, each interval $g_\jjj(\text{ch}(E))$, $\jjj \in \Gamma^h$ which intersects $Q$ contains at least one of its endpoints, and a given endpoint can intersect at most $M^h$ of the intervals $g_\jjj(\text{ch}(E))$, $\jjj \in \Gamma^h$. Note also that each interval $g_\jjj(\text{ch}(E))$ can only intersect at most $2^n r^h + 1 < 1 + r^{-1}$ dyadic intervals $Q \in D_n(\R)$, when $\jjj \in \Gamma^h$.

Let now $\mathcal{E}_h$ denote the partition $\lbrace [\iii]:\ \iii \in \Gamma^h \rbrace$ of $\Gamma^\N$. By the above, for large $n$, each element of the partition $\mathcal{E}_h$ intersects at most $2M^h$ elements of the partition $\pi^{-1}(D_n(\R)) = \lbrace \pi^{-1}(Q):\ Q \in D_n(\R) \rbrace$ and vice versa. Since $H_n(\pi\nu) = H(\nu, \pi^{-1}(D_n(\R)))$, we have by Lemma \ref{lemma almost continuity of entropy} for every large $n$ that
$$
|H_n(\pi\nu) - H(\nu, \mathcal{E}_h)| \leq h \log M + 1,
$$
and so
\begin{align*}
\dim \pi\nu
&\geq -(\log r)^{-1}\liminf_{n \to \infty} \left(\frac{1}{-n (\log r)^{-1}}H(\nu, \mathcal{E}_h) - \frac{h \log M}{-n(\log r)^{-1}} \right) \\
& = \frac{h(\nu, \sigma)}{-\log r} - \frac{\log M}{-\log r}\,,
\end{align*}
as required.
\end{proof}

\section{Invariant measures and the weak separation condition}\label{sec:ws}

In this section, we discuss some properties of projections of invariant measures for IFS's on the line satisfying the weak separation. We recall the definition only in the special case most relevant to our application.

\begin{definition}
Let $\mathcal{F}=\{f_i(x)=rx+\lambda_i\}_{i\in \Gamma}$ be a homogeneous affine IFS on $\R$. We say $\mathcal{F}$ satisfies the {\em weak separation condition} (WSC) if there is $c>0$ such that for any $n\ge 1$ and $\iii,\jjj\in \Gamma^n$,
$$|f_\iii(0) - f_\jjj(0)| =0 \ \textrm{ or } \ |f_\iii(0) - f_\jjj(0)|> cr^n.$$
\end{definition}
An important consequence of the WSC is that for any $a\in \R$ we have
$$\left\{[a,a+r^n]\cap \{f_\iii(0):\iii\in \Gamma^n\}\right\}\le M_1,\ \ n\in \N$$
for some constant $M_1$ depending only on $\cF$.

The following proposition on the absolute continuity of the projections of shift-invariant measures will play a crucial role in the proof of Theorem \ref{thm-principal-1}.

\begin{proposition}\label{prop-wsc-abs-cont}
Let $\mathcal{F}=\{f_i(x)=r x+\lambda_i\}_{i\in\Gamma}$ be a homogeneous affine IFS on $\R$ satisfying the WSC. Let $\nu$ be a shift-invariant measure on $\Gamma^\N$. Then
$$\pi\nu \ll \mathcal{L}^1\iff \dim \pi \nu=1,$$
where $\pi$ denotes the natural coding map from $\Gamma^\N$ to the attractor of $\mathcal{F}$.
\end{proposition}

Special cases of this result are known. Ruiz \cite{Ruiz} proved it (and several additional properties)  in the case where $r^{-1}$ is an integer, the translations $\lambda_i$ are rational, and $\nu$ is a Bernoulli measure. Although he didn't state it in this language, in the special case of Bernoulli convolutions with Pisot parameter, the proposition goes back to Garsia in the 1960s \cite{Garsia}. Our approach is similar to Garsia's, but we emphasize that we require information about every shift-invariant measure on the symbolic space.

Proposition \ref{prop-wsc-abs-cont} is a consequence of the following two lemmas.

\begin{lemma}\label{lemma sub additivity entropy WSC}
Under the assumptions of Proposition \ref{prop-wsc-abs-cont}, there exists a constant $C$ depending only on the IFS $\cF$ such that, denoting $\mu=\pi\nu$,
$$H_{n+m}(\mu)\le H_{n}(\mu)+H_{m}(\mu)+C.$$
\end{lemma}

\begin{lemma}\label{lemma quantitative entropy estimate implies abs cont}
Let $\mu$ be a probability measure on $[0,1]$ such that for some constant $M$,
\begin{equation}\label{eq proposition quantitative entropy estimate 1}
H_n(\mu)\ge n-M \ \ {\rm for \ all\ } \ n.
\end{equation}
Then $\mu\ll \mathcal{L}^1$.
\end{lemma}
In fact, this lemma is due to Garsia \cite{Garsia}, but below we present the short proof for completeness.

We first show how to conclude the proof of Proposition \ref{prop-wsc-abs-cont}.
\begin{proof}[Proof of Proposition \ref{prop-wsc-abs-cont}]
Let $\nu\in \mathcal{M}_{\rm inv}(\Gamma^\N,\sigma)$. Applying Lemma \ref{lemma sub additivity entropy WSC} $k$ times, we have %on each term of the sequence $(\frac{1}{kn} H_{kn}(\pi \nu))_k$ we deduce that
$$H_n(\pi\nu)\ge n\frac{H_{nk}(\pi\nu)}{nk} -C \ \ {\rm for \ all\ } \ n.$$
If  $\dim(\pi\nu)=1$, then $\frac1{nk}H_{nk}(\pi\mu)\longrightarrow 1$ as $k\to\infty$, and so Lemma \ref{lemma quantitative entropy estimate implies abs cont} yields that $\pi\nu$ is absolutely continuous.
\end{proof}

\begin{proof}[Proof of Lemma \ref{lemma sub additivity entropy WSC}]
Again, we may assume that $\diam(E)=1$. Define the sets
$$
\mathcal{A}_n = \lbrace f_\iii(0):\ \iii \in \Gamma^n \rbrace
\ \ {\rm and}\ \
B_n(a) = \lbrace \iii \in \Gamma^n:\ f_\iii(0) = a \rbrace.
$$
Given a collection of finite sequences $\mathcal{I}$, we denote $[\mathcal{I}]=\bigcup_{\iii\in\mathcal{I}}[\iii]$. In particular,
\[
[B_n(a)] = \bigcup\{[\iii]: \iii\in\Gamma^n, f_\iii(0)=a\}.
\]
We define the partitions $\mathcal{P}_n=\{ [B_n(a)]: a\in \mathcal{A}_n\}$. We proceed by first showing that
\begin{equation}\label{eq Garsia entropy subadditive}
H(\nu, \mathcal{P}_{n+m}) \leq H(\nu, \mathcal{P}_n) + H(\nu, \mathcal{P}_m)
\end{equation}
for all integers $n$ and $m$, and then showing that if $h = \lfloor -n (\log r)^{-1} \rfloor$ is the integer for which $r^{h+1}< 2^{-n}\le r^{h}$, then
\begin{equation}\label{eq entropy wsc bound 1}
|H_n(\mu) - H(\nu, \mathcal{P}_h)| \leq M_2
\end{equation}
for some constant $M_2$ independent of $n$. Lemma \ref{lemma sub additivity entropy WSC} then follows by combining \eqref{eq Garsia entropy subadditive} and \eqref{eq entropy wsc bound 1}.

We begin with \eqref{eq Garsia entropy subadditive}.  Note that the partition $\mathcal{P}_n \vee \sigma^{-n}\mathcal{P}_m$ refines $\mathcal{P}_{n+m}$. Indeed, if $\iii \in \Gamma^{n+m}$, then $f_{\iii|_n}(0)$ and $f_{\sigma^n\iii}(0)$ determine $f_{\iii}(0)$. Hence
\begin{align*}
H(\nu, \mathcal{P}_{n+m}) &\le H(\nu,\mathcal{P}_n \vee \sigma^{-n}\mathcal{P}_m) \\
&\overset{\eqref{eq-entropy-refinement}}{\le} H(\nu,\mathcal{P}_n)+H(\nu,\sigma^{-n}\mathcal{P}_m)\\
&\overset{\sigma\nu=\nu}{=} H(\nu,\mathcal{P}_n)+H(\nu,\mathcal{P}_m).
\end{align*}

To prove \eqref{eq entropy wsc bound 1}, fix an integer $n$ and let $h = \lfloor -n (\log r)^{-1} \rfloor$. We now claim that as a consequence of the WSC, there exists a constant $M_2$ depending only on the IFS $\mathcal{F}$ such that each element of the partition $\mathcal{P}_h$ intersects at most $M_2$ elements of the partition
$\pi^{-1}(D_n(\R))$, % = \lbrace \pi^{-1}(Q):\ Q \in D_n(\R) \rbrace$,
and vice versa.

Indeed, for any $\iii \in \Gamma^\N$ the set $\pi[(\iii|_h])$ has diameter $r^h$ and can thus intersect at most $2^n r^h+1 < r^{-1}+1$ dyadic intervals of level $n$. This shows that each element in $\mathcal{P}_h$ can intersect at most $r^{-1}+1$ elements of $\pi^{-1}(D_n(\R))$.
On the other hand, recall that $M_1$ was chosen so that
$$
\lbrace [a, a+ r^n] \cap \lbrace f_\iii(0):\ \iii \in \Gamma^n \rbrace \rbrace \leq M_1
$$
for all $a \in \R$ and $n \in \N$. Now, if $Q \in D_n(\R)$ and $\iii \in \Gamma^\N$ is such that $\pi(\iii) \in Q$, the point $f_{\iii|_h}(0)$ must be within distance $r^h$ of $Q$. Since the number of level-$n$ dyadic intervals within this distance of $Q$ is at most $2^{n}r^h + 1 < r^{-1} + 1$, and each of them can contain $M_1$ distinct points $f_\iii(0)$, $\iii \in \Gamma^h$, we know that any set in $\pi^{-1}(D_n(\R))$ can intersect at most $M_1(1+ r^{-1})$ cylinders of $\mathcal{P}_h$.

Thus, we may choose $M_2 = M_1(1+r^{-1})$, and \eqref{eq entropy wsc bound 1} follows by using Lemma \ref{lemma almost continuity of entropy} and the equality $H_n(\mu) = H(\nu, \pi^{-1}(D_n(\R)))$.
\end{proof}

\begin{proof}[Proof of Lemma \ref{lemma quantitative entropy estimate implies abs cont}]
Suppose that $\mu \not\ll \mathcal{L}^1$ and let $E$ be a set, which we may assume compact, such that $0 < \mu(E) := C$ and $\mathcal{L}^1(E) = 0$. Then for every $\varepsilon > 0$ there exist $n$ and a set $F\supset E$ which is a union of $<\varepsilon 2^n$ dyadic intervals of length $2^{-n}$. We have
\begin{align*}
H_n(\mu) &= H_n(\mu(F)\mu_F+(1-\mu(F))\mu_{[0,1]\setminus F}) \\
&\overset{\eqref{eq-entropy-convexity}}{\le}  \mu(F)H_n(\mu_F)+ (1-\mu(F)) H_n(\mu_{[0,1]\setminus F}) +1\\
&\le \mu(F) \log(\varepsilon 2^n) + (1-\mu(F))n + 1\\
&\le C\log(\varepsilon) + n + 1.
\end{align*}
Since $\varepsilon$ is arbitrarily small, \eqref{eq proposition quantitative entropy estimate 1} cannot hold, as claimed.
\end{proof}

Although this is not needed for the proof of Theorem \ref{thm-principal-1}, we remark that our proof also establishes the following generalization of Proposition \ref{prop-wsc-abs-cont}:

\begin{proposition}\label{prop max dim implies abs cont}
Let $\mathcal{F} = \lbrace f_i \rbrace_{i \in \Gamma}$ be a homogeneous affine IFS on $\R$ satisfying the WSC, and let $K$ denote its attractor. Let $\nu\in \mathcal{M}_{\rm inv}(\Gamma^\N,\sigma)$ and $\mu = \pi \nu$, where $\pi$ denotes the natural coding map from $\Gamma^\N$ to $K$. If $\dim \mu= \dim K = s$, then $\mu$ is absolutely continuous with respect to the $s$-dimensional Hausdorff measure.
\end{proposition}

Indeed, it was shown in \cite{FHOR} that self-similar sets under the WSC are Ahlfors-regular. A small variant of the proof of Lemma \ref{lemma quantitative entropy estimate implies abs cont}, using generalized dyadic partitions of $K$ in place of $\mathcal{D}_n$ (see e.g. \cite{KRS}), shows that $\mu\ll \mathcal{H}^s$ provided $\mu$ is a probability measure on an Ahlfors regular set $E\subset [0,1]$ with dimension $s$ such that for some constant $M$,
\[
H_n(\mu)\ge ns-M \ \ {\rm for \ all\ } \ n.
\]
Moreover, the proof of Lemma \ref{lemma sub additivity entropy WSC} also goes through in this setting.
We note that Ahlfors regularity gets used in the proof of Lemma \ref{lemma quantitative entropy estimate implies abs cont} for the estimate $H_n(\mu_{[0,1]\setminus F})\le sn+C$, as well as to conclude the existence of the number $M_2$ in the proof of Lemma \ref{lemma sub additivity entropy WSC}.

With minor technical additional complications, Proposition \ref{prop max dim implies abs cont} extends to arbitrary self-similar systems on $\R^d$ satisfying the WSC. The details are left to the interested reader.

We also require the following semicontinuity result for the dimension of invariant measures on a set satisfying the WSC. It is a corollary of \cite[Theorem 2.8 and Proposition 4.20]{FH}, but for the reader's convenience we give a short proof for it.

\begin{lemma}\label{wsc dim cont}\label{lemma-5}
Let $\mathcal{F}=\{f_i\}_{i\in \Gamma}$ be a homogeneous affine IFS on $\R$ satisfying the WSC. Then the map
$$\mu\mapsto \dim \pi\mu$$
is upper semi-continuous (with respect to the weak$^*$-convergence) in the set of invariant measures on $\Gamma^{\mathbb{N}}$, where $\pi$ denotes the natural coding map from $\Gamma^\N$ to the attractor of $\mathcal{G}$.
\end{lemma}

\begin{proof}[Proof of Lemma \ref{wsc dim cont}]
Let $(\mu_k)_k$ be a sequence of arbitrary invariant probability measures in $\Gamma^\N$ converging to a measure $\mu$. Fix an integer $n$ and define $H_n'(\pi\mu) = \int_0^1 H_n(\pi\mu(\cdot + x)) \,dx$. Observe that for every $Q \in D_n(\R)$ and for $\mathcal{L}^1$-almost every $x$, the set $Q + x$ is a set of continuity for $\pi\mu$, that is, $\pi\mu(\partial (Q+x))=0$. Indeed, if this was not the case, there would exist a $c > 0$ and an infinite set $\lbrace x_1, x_2, \ldots \rbrace$ such that $\pi\mu(x_i) > c$ for every $i$, contradicting the finiteness of $\mu$.

Thus, $\lim_{k \to \infty} \pi\mu_k(Q + x) = \pi\mu(Q + x)$ for every $Q$ and almost every $x$, and hence
$$
\lim_{k \to \infty} H_n'(\pi \mu_k) = H_n'(\pi \mu).
$$
Since $(\mu_k)_k$ was an arbitrary converging sequence, the mapping $\mu \mapsto H_n'(\pi\mu)$ is thus continuous in the invariant probability measures of $\Gamma^\N$. On the other hand, by Lemma \ref{lemma almost continuity of entropy} we have $$|H_n(\pi\mu) - H_n(\pi\mu(\cdot + x))| \leq 1$$ for all $x \in \R$. As a consequence of this fact and Lemma \ref{lemma sub additivity entropy WSC}, for any $\mu \in \cM_{\rm inv}(\Lambda^\N, \sigma)$ we may write
$$
\dim \pi\mu = \inf_{n \in \N} \frac{H_n'(\pi\mu)+ C'}{n}
$$
for some constant $C'$. Indeed, if $C$ is the constant from Lemma \ref{lemma sub additivity entropy WSC} and if $\mu\in \cM_{\rm inv}(\Lambda^\N, \sigma)$ satisfies $\frac{H_n'(\pi\mu)+ C+1}{n}<\dim\pi\mu-\varepsilon$, this  implies that
\begin{align*}
\dim(\pi\mu)=\liminf_{k\to\infty}\frac{H_{nk}(\pi\mu)}{nk}\le \frac{H_n(\pi\mu)}{n}+\frac{C}{n}\le \frac{H'_{n}(\pi\mu)}{n}+\frac{C+1}{n}<\dim(\pi\mu)-\varepsilon\,,
\end{align*}
which is absurd. Thus, $\mu \mapsto \dim \pi\mu$ is an infimum of continuous functions and, as such,  upper semi-continuous.
\end{proof}

\section{Proof of Theorem \ref{thm-principal-1}}\label{sec:proof}
\subsection{Projections of invariant measures}\label{ss:four}

Fix an integer $N\ge 2 $ and let
$\Gamma\subset\{0,\ldots,N-1\}^d$
such that $\#\Gamma < N^d$, and consider the homogeneous IFS on $\R^d$ defined by
\begin{equation} \label{eq-Nadic-IFS}
\cF=\left\{f_i(x)=rx+\lambda_i\right\}_{i \in \Gamma}\,,
\end{equation}
where $r=\frac1N$ and $\lambda_i=\frac{i}{N}\in\R^d$.
Let $K$ be the attractor of $\cF$. Given any closed $T_N$-invariant set $L\subsetneq [0,1]^d$ we can find $q$ such that not all words in $(\{0,\ldots,N-1\}^d)^q$ appear in $L$ under the natural symbolic coding. If we let $K$ be the self-similar set as above corresponding to $N^q$ and $\Gamma$ in correspondence with the words of length $q$ that appear in $L$, then $L\subset K\subsetneq [0,1]^d$. Hence it is enough to prove the claim of Theorem \ref{thm-principal-1} for self-similar $T_N$-invariant sets $K$ corresponding to an IFS of the form \eqref{eq-Nadic-IFS}.

For a vector $0\neq v = (v_1, \ldots, v_d) \in \R^d$, we define the linear projection
\[
P_v(x) = \langle x,v\rangle: \R^d\to\R.
\]
Note that we do not assume that $v$ has unit norm. Up to a scaling, and identifying $l_v$ with the real line through a linear isomorphism, this coincides with the orthogonal projection onto the line $l_v = \lbrace tv:\ t \in \R \rbrace$.

We denote the projection of $\cF$ under $P_v$ by
$$
\cF_v = \lbrace f_i^v(t) = r t + P_v(\lambda_i)\rbrace_{i \in \Gamma}
$$
and the projection of $K$ under $P_v$ by $K_v = P_v(K)$. Observe that $K_v$ is then the attractor of $\cF_v$.

A key observation where we begin is the fact that for any $T_N$-invariant measure on $K$, there exists at least one direction in which the projection of the measure is not absolutely continuous. Analysing the Fourier transform of $\mu$,
$$
\widehat{\mu}(\xi) = \int e^{-2 \pi i \langle\xi,  x\rangle}\,d\mu(x),
$$
for $\xi \in \Z^d$, this is a simple consequence of the uniqueness of Fourier coefficients and the Riemann-Lebesgue lemma.  We let $Z(R_0)$ denote the set $B(0,R_0)\cap \Z^d\setminus \{0\}$. Our set of directions $\mathcal{V}$ will be the set $Z(R_0)$ provided by the following lemma.

\begin{lemma}\label{proj sing corollary}\label{corollary-2}
For each $0\neq v\in\Z^d$, and each $\mu \in \cM_{\rm inv}(K, T_N)$, either $P_v\mu \bmod 1$ is Lebesgue measure on $[0,1]$, or $P_v \mu \not\ll \mathcal{L}^1$.

Moreover, there exists $R_0 = R_0(K)$ such that for each $\mu \in \cM_{\rm inv}(K, T_N)$, there exists $v \in Z(R_0)$ such that
$
P_v \mu \not\ll \mathcal{L}^1
$.
\end{lemma}

\begin{proof}[Proof of Lemma \ref{proj sing corollary}]
Since $\#\Gamma<N^d$, the compact set $K$ has zero Lebesgue measure, and hence the Lebesgue measure on $[0,1]^d$ is not a weak$^*$ accumulation point of  $\cM_{\rm inv}(K, T_N)$. Since on $[0,1]^d$, weak$^*$ convergence is equivalent to the pointwise convergence of Fourier transform at integer frequencies (see \cite[(3.66)]{MattilaFourier}), there exists $R_0 = R_0(K)$ such that for any $T_N$-invariant $\mu$ on $K$, there exists $v \in Z(R_0)$ such that
$$
\hat{\mu}(v) \neq 0.
$$
A simple computation using the $T_N$-invariance of $\mu$ shows that for any $k \in \N$,
$$
\widehat{P_v\mu}(N^k) = \widehat{P_v\mu}(1)= \hat{\mu}(v),
$$
and the statement follows from an application of the Riemann-Lebesgue lemma.
\end{proof}
It is not hard to derive explicit bounds on $R_0$, see Section \ref{sec:R_0}. Throughout the rest of the section, let us now fix $R_0$ as in Lemma \ref{proj sing corollary}.

In fact, we can say something much stronger about the projected invariant measures. Because of the $N$-adic structure of $K$, the projections $K_v$ in rational directions always satisfy the weak separation condition, and this additional regularity allows us to bound the dimension of the projected measure uniformly away from one, in at least one direction.

\begin{lemma}\label{wsc satisfied}\label{lemma-3}
For any $0\neq v\in \Z^d$, the IFS $\cF_v$ satisfies the WSC (with $c=1)$.
\end{lemma}
\begin{proof}
This was already observed in \cite{Ruiz}, but we include the short deduction. For any $n \in \N$ and $\iii=(i_1,\ldots,i_n), \jjj=(j_1,\ldots,j_n) \in \Lambda^n$, we have
\begin{align*}
f_{\iii}^v(0) - f_{\jjj}^v(0) &= \sum_{\ell=0}^{n-1} N^{-l} (f_{i_\ell}^v(0) - f_{j_\ell}^v(0)) \\&= N^{-n} \sum_{\ell=0}^{n-1} N^{n-l-1} v \cdot (N \lambda_{i_\ell}  -N\lambda_{j_\ell}).
\end{align*}
Since the last sum above is an integer, we see that the WSC is satisfied for $c=1$.
\end{proof}

%
%
%
%
\begin{comment}
Since $\lambda_i \in \Lambda$ are rational, the orbits
$$
\text{Orb}(v \cdot (\lambda_1 - \lambda_2)) = \lbrace m v \cdot (\lambda_1 - \lambda_2) \mod 1:\ m \in \N \rbrace
$$
are finite. Let $\lbrace a_1, \ldots, a_M \rbrace$ be the union of all the orbits taken over $\lambda_1, \lambda_2 \in \Lambda$. Again, the numbers $a_i$ are rational, hence the fractional parts of the elements of the sequence $\left( \sum_{\ell=0}^{n-1} a_{i_\ell} \right)$ belong to the set $\lbrace b_1 + \ldots + b_M:\ b_j \in \text{Orb}(a_j) \rbrace$ which is finite. Finally, if we denote by $d_v$ the smallest absolute value of the non-zero elements of this set, we may deduce that
$$
|f_{\lambda_1}^v(0) - f_{\lambda_2}^v(0)| = 0\ \text{or}\ |f_{\lambda_1}^v(0) - f_{\lambda_2}^v(0)| \geq N^{-n} \Vert v \Vert^{-1} d_v,
$$
so the WSC is satisfied with constant $c_v := \Vert v^{-1} \Vert d_v$.
\end{comment}
%
%
%
%
%

\begin{proposition}\label{prop:dimension-drop}
There exists a constant $\delta_0=\delta_0(K)>0$ such that for any $\mu\in \cM_{\rm inv}(K,T_N)$, there is $v\in Z(R_0)$ such that
$$\dim P_v\mu\le 1-\delta_0.$$
\end{proposition}

Several related problems have been considered in the literature. B{\'a}r{\'a}ny and Rams \cite{BR} showed (in the planar case, and using a different technique based on random matrix products) that in the case in which $N$ does not divide $\#\Gamma>N$, the dimension of the natural self-similar measure on $K$ drops when projected in \emph{any} rational direction. It is easy to see that the result of B{\'a}r{\'a}ny and Rams fails even for other self-similar measures. In a different direction, it follows from a recent result of Jordan and Rapaport \cite[Theorem 1.1]{JR} that if $\mu\in\cM_{\rm inv}(K, T_N)$ has dimension $\ge 1$, then \emph{every} projection of $\mu$ in an irrational direction has dimension $1$. Finally, the related but different problem of whether there exists a non-principal direction $v$ such that $\dim P_v\mu <\dim\mu$ has implications in equidistribution theory; see \S\ref{ss:dim-drop-non-principal} for further discussion on this connection.

\begin{proof}[Proof of Proposition \ref{prop:dimension-drop}]
Suppose otherwise. Then we can find a sequence of $T_N$-invariant measures $\mu_k$ on $K$ with $\dim P_v\mu_k\to 1$
for all $v\in Z(R_0)$. Passing to a subsequence, we may assume that $\mu_k$ converges to some measure $\mu_\infty\in \cM_{\rm inv}(K,T_N)$.
By Lemma \ref{lemma-3}, for each $v\in Z(R_0)$ the projected IFS $\cF_v$ satisfies the WSC.  We thus deduce from Lemma \ref{lemma-5} that
\begin{equation}\label{eq: lemma-5, 1}
\dim P_v \mu_\infty\ge \limsup_k\dim P_v\mu_k=1 \textrm{ for all  } v\in Z(R_0).
\end{equation}
On the other hand, from Lemma \ref{proj sing corollary} we know that there is $v'\in Z(R_0)$ such that
$$P_{v'}\mu_\infty \not\ll \mathcal{L}^1.$$
Since the IFS $\cF_{v'}$ satisfies the WSC, applying Proposition \ref{prop-wsc-abs-cont}, we get
$$\dim P_{v'}\mu_\infty<1.$$
This is a contradiction to \eqref{eq: lemma-5, 1}.
\end{proof}

In fact, the same bound for the dimension works for measures invariant under any iteration of $T_N$.

\begin{corollary}\label{corollary-7}
For any $m\ge 1$ and $\mu\in \cM_{\rm inv}(K,T_N^m)$, there exists $v\in Z(R_0)$ such that
$$\dim P_v\mu\le 1-\delta_0.$$
\end{corollary}

\begin{proof}
Let us fix $m\ge 1 $ and $\mu\in \cM_{\rm inv}(K,T_N^m)$. It is readily checked that
$$\nu=\frac{1}{m}\sum_{k=0}^{m-1}T_N^k\mu$$
is $T_N$-invariant. Applying Proposition \ref{prop:dimension-drop} to $\nu$, we obtain a $v\in Z(R_0)$ such that
$$\dim P_v\nu\le 1-\delta_0.$$
Since $P_v(T_N^k\mu)$ is a finite sum of translated and rescaled copies of $P_v \mu$, a straightforward calculation using the concavity of the function $\mu \mapsto H_m(\mu)$ and Lemma \ref{lemma almost continuity of entropy} shows that $\dim P_v(T_N^k\mu)=\dim P_v(\mu)$ for every $k \geq 1$. Thus, we must have
$$\dim P_v\mu\le 1-\delta_0\,,$$
as required.
\end{proof}

\subsection{The high-level IFS}\label{ss:hl}

Since we aim to pass from dimensions of invariant measures to dimensions of sets through an application of the variational principle, we require a bound for the entropies of the invariant measures, given by Lemma \ref{entropy dimension}. To get the error term in the statement of the lemma negligible, we need to inspect a sufficiently high-level iteration of $\cF$.

Let us fix a large integer $m$ with
\begin{equation}\label{eq: high-level-number, 1}
\frac{\log(2\sqrt{d}R_0)}{m\log N} <\frac{\delta_0}{2}.
\end{equation}
Later we will see why we make the choice of such an $m$. By replacing the original IFS by its high-level iteration $\{f_\iii:\iii\in\Gamma^m\}$ (which has contraction ratio $N^{-m}$ and the same attractor $K$), we assume from now on that $m=1$. In order for this change to preserve \eqref{eq: high-level-number, 1}, it is crucial for us that $\delta_0$ and $R_0$ do not depend on $m$, a fact guaranteed by Corollary \ref{corollary-7}.

We now proceed to remove any copies of a single contraction in $\cF_v$ caused by exact overlaps. Define an equivalence relation $\sim_v$ in $\Gamma$ by $i\sim_v j$ if and only if $f_{i}^v=f_{j}^v$. For $i \in \Gamma$, we denote its equivalence class under $\sim_v$ by $[i]_v$, and let $f_{[i]_v}=f_i$. Consider the collection of equivalence classes,
\begin{equation}\label{eq: equivalent class, 1}
\Gamma_v=\{[i]_v:i\in \Gamma\}\,.
\end{equation}
When there is no danger of misunderstanding, we drop the brackets and denote the equivalence classes simply by $i,j$, etc.

For each $v\in Z(R_0)$, let us consider the IFS
$$
\widetilde{\cF}_v=\left\{f_{i}^v:i\in \Gamma_v\right\}.
$$
Note that $K_v$ is also the attractor of $\widetilde{\cF}_v$. We write $\pi_v$ for the natural coding map from the symbolic space $\Gamma_v^\N$ to the attractor $K_v$.

Observe that for each $i\in \Gamma_v$, the map $f_{i}^v$ has contraction ratio $N^{-1}$, and for
$i,j \in \Gamma_v$ with $i\neq j$, Lemma \ref{wsc satisfied} implies that
$$
\left|f^v_{i}(0) - f^v_{j}(0)\right|\ge N^{-1}.
$$
On the other hand, denoting the convex hull of $K_v$ by $I_v$,
\[
\diam(f^v_{i}(I_v)) = \diam(I_v)N^{-1} \le \diam([0,1]^d)\|v\|N^{-1} \le \sqrt{d}R_0 N^{-1}.
\]
We deduce that for each $i\in \Gamma_v$,
\begin{equation}\label{eq: distance-cylinder-high-level, 1}
\left|\left\{j\in \Gamma_v:f^v_{j}(I_v)\cap f^v_{i}(I_v) \neq \emptyset\right\}\right|\le 2\sqrt{d}R_0.
\end{equation}
Using \eqref{eq: distance-cylinder-high-level, 1}, recalling \eqref{eq: high-level-number, 1}, and applying Lemma \ref{lemma-8} to the IFS $\widetilde{\cF}_v$, we get:
\begin{proposition}\label{proposition-9}
If $\mu\in \cM_{\rm inv}\left(\Gamma_v^\N,\sigma\right)$, then
$$\dim \pi_v \mu \ge \frac{h(\mu,\sigma)}{\log N}-\frac{\log(2\sqrt{d}R_0)}{\log N} \ge \frac{h(\mu,\sigma)}{\log N}-\frac{\delta_0}{2}.$$
\end{proposition}

\subsection{Covering of $K$}\label{sec:final_proof}
We now form a cover for the symbolic space $\Gamma^\N$, the natural projection of which then serves as a cover of the desired type for the attractor $K$. For each $v\in Z(R_0)$, let
$$B_v=\left\{\mu\in \cM_{\rm inv}\left(\Gamma^\N,\sigma\right): \dim P_v (\pi \mu)\le 1-\delta_0\right\}.$$
Since $\pi\mu$ is $T_N$-invariant, Corollary \ref{corollary-7} implies that
\begin{equation}\label{eq: cover-partition, 1}
\cM_{\rm inv}\left(\Gamma^\N,\sigma\right)\subset \bigcup_{v\in Z(R_0)} B_v.
\end{equation}
Recall that $V(\iii)$ denotes the collection of the weak$^*$ accumulation points of the sequence
$
\frac{1}{n}\sum_{k=0}^{n-1}\delta_{\sigma^{k}(\iii)}
$
(we use this notation on $\Gamma^\N$ as well as on the symbolic spaces $\Gamma_v^\N$).
It is readily checked that $V(\iii) \subset \cM_{\rm inv}(\Gamma^\N, \sigma)$. Note also that $V(\iii) \neq \emptyset$ by the compactness of the space of probability measures. For each $v\in Z(R_0)$, let
$$
D_v=\left\{\iii \in \Gamma^\N: \exists \mu\in V(\iii) \textrm{ such that } \mu\in B_v\right\}.
$$
It follows from \eqref{eq: cover-partition, 1} that
$$\Lambda^\N\subset \bigcup_{v\in Z(R_0)} D_v. $$
It remains to show that for each $v\in Z(R_0)$, we have
$$\dimh P_v(\pi D_v)<1.$$

Let us consider the projection map $\Pi_v:\Gamma^\N\to \Gamma_v^\N$ defined by
$$
\Pi_v(i_1, i_2, \ldots)=([i_1]_v, [i_2]_v, \ldots)\,.
$$
%for $\iii=(i_1,i_2,\ldots) \in \Gamma^\N$.
Notice that, by definitions of $\pi, \pi_v$ and $\Pi_v$, we have for each $\iii \in \Gamma^\N$,
\begin{equation}\label{eq: relation-between-coding-map,1}
P_v(\pi(\iii))=\pi_v(\Pi_v(\iii)).
\end{equation}
Thus, we only need to show that for each $v\in Z(R_0)$,
\begin{equation}\label{eq:claim_reduced}
\dimh \pi_v(\Pi_v D_v)<1\,.
\end{equation}

Let us fix a $v\in Z(R_0)$. For each $\iii \in D_v$, there exists a sequence $(n_k)_k$ and $\mu\in B_v$ such that
$$\frac{1}{n_k}\sum_{j=0}^{n_k-1}\delta_{\sigma^j(\iii)}\rightharpoonup \mu.$$
Since the map $\Pi_v$ is continuous and $\Pi_v\circ \sigma=\sigma\circ \Pi_v$\footnote{Note that in the equation $\Pi_v\circ \sigma=\sigma\circ \Pi_v$, the first $\sigma$ is the shift map on the symbolic space $\Gamma^\N$ and the second $\sigma$ denotes the shift map on the space $\Gamma_v^\N$.}, the above convergence implies that
\begin{equation} \label{eq:accumulation-point-under-Qv}
\frac{1}{n_k}\sum_{j=0}^{n_k-1}\delta_{\sigma^j(\Pi_v(\iii))}\rightharpoonup \Pi_v\mu\in \cM_{\rm inv}\left(\Gamma_v^\N,\sigma\right).
\end{equation}

Since $\mu\in B_v$, by the definition of $B_v$ and \eqref{eq: relation-between-coding-map,1}, we know that
$$
\dim \pi_v(\Pi_v\mu)=\dim P_v\left(\pi\mu\right)\le 1-\delta_0.
$$
Applying Proposition \ref{proposition-9}, we get
$$
\frac{h\left(\Pi_v\mu,\sigma\right)}{\log N}\le 1-\delta_0+\delta_0/2=1-\delta_0/2.
$$
In particular, we have
$$
h(\Pi_v\mu,\sigma)\le (1-\delta_0/2)\log N.
$$

Keeping \eqref{eq:accumulation-point-under-Qv} in mind, we have thus proved that
$$
\Pi_v(D_v)\subset \left\{\iii\in \Gamma_v^\N : \exists \nu\in V(\iii) \textrm{ such that } h(\nu,\sigma)\le  (1-\delta_0/2)\log N\right\}=:E_v.
$$
Using Lemma \ref{lemma-Bowen,1} for the covering sets $E_v$, we thus have
$$h_{\rm top} (\Pi_v(D_v),\sigma)\le h_{\rm top} (E_v,\sigma)\le (1-\delta_0/2)\log N.$$
On the other hand, since  $\widetilde{\cF}_v$ is a homogeneous IFS with contraction ratio $N^{-1}$, we deduce from \eqref{eq:dimension-from-entropy} and the above inequality that
$$\dimh\pi_v(\Pi_vD_v)\le \frac{h_{\rm top} (\Pi_v(D_v),\sigma)}{\log N}\le 1-\delta_0/2<1\,.$$
We have thus verified \eqref{eq:claim_reduced} and this completes the proof of Theorem \ref{thm-principal-1}.

\section{Remarks and further results}\label{sec:rems}

\subsection{Homogeneous self-similar sets with no rotations and no grid structure}
 In this section we prove Proposition \ref{prop: homo-small}. To this end, let $\Gamma=\{0,\ldots,m-1\}$, $r>0$, let $\mathcal{F}=\{f_i(x)=r x+\lambda_i\}_{i\in\Gamma}$ be a homogeneous IFS with attractor $K\subset\R^d$, and let $\pi\colon\Gamma^\N\to K$ be the natural projection. Given $0\neq v\in\R^d$, let us also denote by $\mathcal{F}_v$ the projected IFS  $\mathcal{F}_v=P_v\circ\mathcal{F}=\{f^v_i(x)=r x+P_v(\lambda_i)\}_{i\in\Gamma}$
  and let $K=K_v\subset\R$ denote its attractor.

Our proof is based on the fact that for any two symbols $i$ and $j$, some projection induces exact overlaps for the pieces  $f_i(K)$ and $f_j(K)$. We note that Harangi's proof \cite{Harangi} for the tube-nullity of the Koch curve is also based on certain exact overlaps, but there the geometric situation is more complicated due to the involved rotations.

To be slightly more precise, for each pair of distinct elements $i,j\in\Lambda$, we can fix a direction $v=v_{i,j}$ such that $f^v_i=f^v_j$. Our set of direction is then given by
\[\mathcal{V}=\{v_{i,j}\,:\,i,j\in\Gamma\,,i\neq j\}\,.\]
In particular, we have $\#\mathcal{V}=\frac{m(m-1)}{2}$.

Given $v=v_{i,j}\in\mathcal{V}$, we consider the reduced IFS
$\widetilde{\cF}_v=\{f^v_l(x)=r x+P_v(\lambda_l)\}_{l\in\Gamma_v}$, where $\Gamma_v=\Gamma\setminus\{j\}$. Due to the exact overlap $f^v_{i}(0)=f^v_j(0)$, it follows that $K_v$ is also the attractor of $\widetilde{\cF}_v$. We denote by $\pi_v$ the natural projection $\Gamma_v^\N\to K_v$. Let us also denote by $\Pi_v\colon\Gamma^\N\to\Gamma_v^\N$ the projection map
that replaces each symbol $j$ by the symbol $i$ in each word $\iii\in\Gamma^\N$.

If $r<\frac1{m-1}$, then the similarity dimension of $K_v$ equals $\frac{\log(m-1)}{-\log r}<1$ and thus also $\dimh(P_v(K))<1$.
 The interesting case is when $\frac{\log{m}}{-\log{r}}$, the similarity dimension of $K$, is between $\frac{\log m}{\log(m-1)}$ and $d(m)=\frac{\log m}{\log m-\frac2m}$, in which case it is possible that all projections of the set $K$ have positive Lebesgue measure.

We will need the following elementary fact whose proof is a simple exercise. For any $i,j\in\Gamma$ and $\kkk=(k_1,k_2,\ldots)\in\Gamma^\N$, let us denote
$$F_{i,j}(\kkk)=\limsup_{n\to\infty}\frac{1}{n}\left|\{0\le l\le n-1:k_l\in \{i,j\}\}\right|\,.$$
\begin{lemma}\label{lemma: frequence1}
For each $\iii\in \Gamma^\N$, there exists distinct $i,j\in\Gamma$ such that $F_{i,j}(\iii)\ge 2/m.$
\end{lemma}
%
%
%
\begin{comment}
\begin{proof}
Let us fix $x\in \Lambda^\N$. For each $n\ge 1$ and $i\in \Lambda$, set $p_n(i)=|\{0\le k\le n-1:x_k=i\}|/n$. We observe that for each $n\ge 1$, $\sum_{i\in \Lambda}p_n(i)=1.$  From this we deduce that for all $n\ge 1$, there exists at least one $(i,j)\in D(m)$ such that
\begin{equation}\label{eq: lemma:freq 1}
p_n(i)+p_n(j)\ge \frac{2}{m}.
\end{equation}
In turn, this implies the existence of a pair $(i,j)\in D(m)$ such that \eqref{eq: lemma:freq 1} occurs for infinitely many $n\in \N.$ It then follows that for this pair $(i,j)$ we have $F_{i,j}(x)\ge 2/m.$
\end{proof}
\end{comment}
%
%
%

Let us now explain how we may use Lemma \ref{lemma: frequence1} to finish the proof. Given $v=v_{i,j}$, let us consider the set
$$B_{v}=\left\{\iii\in \Gamma^\N: F_{i,j}(\iii)\ge \frac2m\right\}\,.$$
Lemma \ref{lemma: frequence1} implies that $\Gamma^\N\subset\cup_{\mathcal{V}}B_v$ and thus it is enough to show that for each $v$, we have
\begin{equation}\label{eq:claim_reduced}
\dimh(P_v(\pi B_v))\le\frac{\log m-2/m}{-\log r}\,.
\end{equation}
Here comes the main observation: On one hand, $P_v(\pi B_v)=\pi_v(\Pi_v B_v)$, and on the other hand,
\[\Pi_v(B_v)=\left\{(k_1,k_2,\ldots)\in\Gamma_v^\N\,:\,\limsup_{n\to\infty}\frac{1}{n}\left|\{0\le l\le n-1:k_l=i\}\right|\ge\frac{2}{m}\right\}\,.\]
Thus, we have reduced the problem to a variant of a classical Besicovitch-Eggleston type problem of bounding the Hausdorff dimension of a set defined using digit frequencies

We fix $v=v_{i,j}\in\mathcal{V}$ and proceed to estimate the dimension of $\pi_v(\Pi_v B_v)$ via the topological entropy of $\left(\Pi_v B_v,\sigma\right)$.

Given $\kkk=(k_1,k_2,\ldots)\in\Pi_v B_v$, there exists a sequence $(n_\ell)_\ell$ such that
\begin{equation}\label{eq: lemma:freq 2}
\liminf_{\ell\to\infty}\frac{1}{n_\ell}|\{0\le q\le n_\ell-1: k_q=i\}|\ge \frac{2}{m}.
\end{equation}
Passing to a subsequence, we may (and do) assume that for some measure $\mu\in \mathcal{M}_{\rm inv}\left(\Gamma_v^\N,\sigma\right)$
$$\frac{1}{n_\ell}\sum_{q=0}^{n_\ell-1}\delta_{\sigma^q(\kkk)} \  \rightharpoonup \mu \textrm{ as } \ell\to \infty$$
in the weak$^*$ topology of measures on $\Gamma_v^\N$. From \eqref{eq: lemma:freq 2}, we deduce that the measure $\mu$ satisfies
$\mu[i]\ge \frac{2}{m}$ and thus
$$h(\mu,\sigma)\le -\sum_{l\in \Gamma_v}\mu[l]\log \mu[l]\le \frac{2}{m}\log \frac{m}{2}+\frac{m-2}{m}\log m=\log m-\frac{2}{m}\,.$$
Thus we have shown that
$$\Pi_v(B_v)\subset \left\{\kkk\in \Gamma_v^\N : \exists \mu\in V(\kkk) \textrm{ such that } h(\mu,\sigma)\le  \log m-\frac{2}{m}\right\}=:E_v\,.$$
Applying Lemma \ref{lemma-Bowen,1} to $E_v$, we get
$$h_{\rm top} (\Pi_v(B_v),\sigma)\le h_{\rm top} (E_v,\sigma)\le \log m-\frac{2}{m}\,.$$
Finally,  since  $\widetilde{\mathcal{F}}_{v}$ is a homogeneous IFS with contraction ratio $r$,  we deduce from \eqref{eq:dimension-from-entropy} and the above inequality %and the hypothesis of Proposition \ref{prop: homo-small}
that
$$\dimh\pi_{v}\left(\Pi_v B_v\right)\le \frac{h_{\rm top} (\Pi_v B_v,\sigma)}{-\log r}\le\frac{\log m-2/m}{-\log r}<1\,.$$

\subsection{Planar self-similar sets with irrational rotations}
\label{ss:rot}
For a  set  $E\subset \R^2$, we define the {\em tube dimension} of $E$, denoted $\dim_{\rm T}E$, to be the infimum of $s>0$ such that for every $\varepsilon>0$ there exists a countable family of tubes $\{T_i\}_i$ such that $E\subset \bigcup_iT_i$ and $\sum_i w(T_i)^s<\varepsilon$. It is clear that every bounded set in $\R^2$ has tube dimension $\le 1$. By Theorem \ref{thm-principal-1} and Proposition \ref{prop: homo-small}, we see that many homogeneous self-similar sets without rotations have simultaneously  Hausdorff dimension $>1$ and tube dimension $<1$. In view of the following Proposition, which follows from the work of the second author \cite{Shmerkin19}, there is a striking difference between self-similar sets with and without rotations with regard to the tube dimension.
\begin{proposition}\label{prop:Pablo}
Let $K$ be a planar self-similar set corresponding to an IFS with infinite rotation group. Then $\dim_{\rm T}(K)=\min(1,\dimh(K))$.
\end{proposition}
\begin{proof}
We may assume that $K$ is a homogeneous self-similar set and it satisfies the open set condition, since otherwise we can apply \cite[Proposition 6]{PeSh} (plus a small additional argument that can be found in \cite[Proof of Theorem 2]{PeSh}) to get a subset of $K$ which satisfies these properties and has Hausdorff dimension arbitrarily close to that of $K$.  It follows from \cite{Shmerkin19} that for each $\varepsilon$ there is $C=C(\varepsilon)>0$, such that for each tube $T$ of width $w$, the intersection $K\cap T$ can be covered by $C w^{-\varepsilon}$ balls of radius $w$ if $\dimh(K)\le 1$, and by $C w^{1-\dimh K-\varepsilon}$ balls of radius $w$ if $\dimh(K)>1$. This follows by combining \cite[Lemma 1.7, Lemma 1.8 and Theorem 8.2]{Shmerkin19} applied to the uniform self-similar measure on $K$. From this, it is not hard to show that the tube dimension of $K$ is at least $\min(1,\dimh(K))$; since the opposite bound is obvious, this completes the proof.
\end{proof}

\begin{remark}
It remains a challenging open problem to determine if all the self-similar sets in Proposition \ref{prop:Pablo} are non tube-null.
\end{remark}

\begin{figure}[H]
	\includegraphics[width=0.9\textwidth]{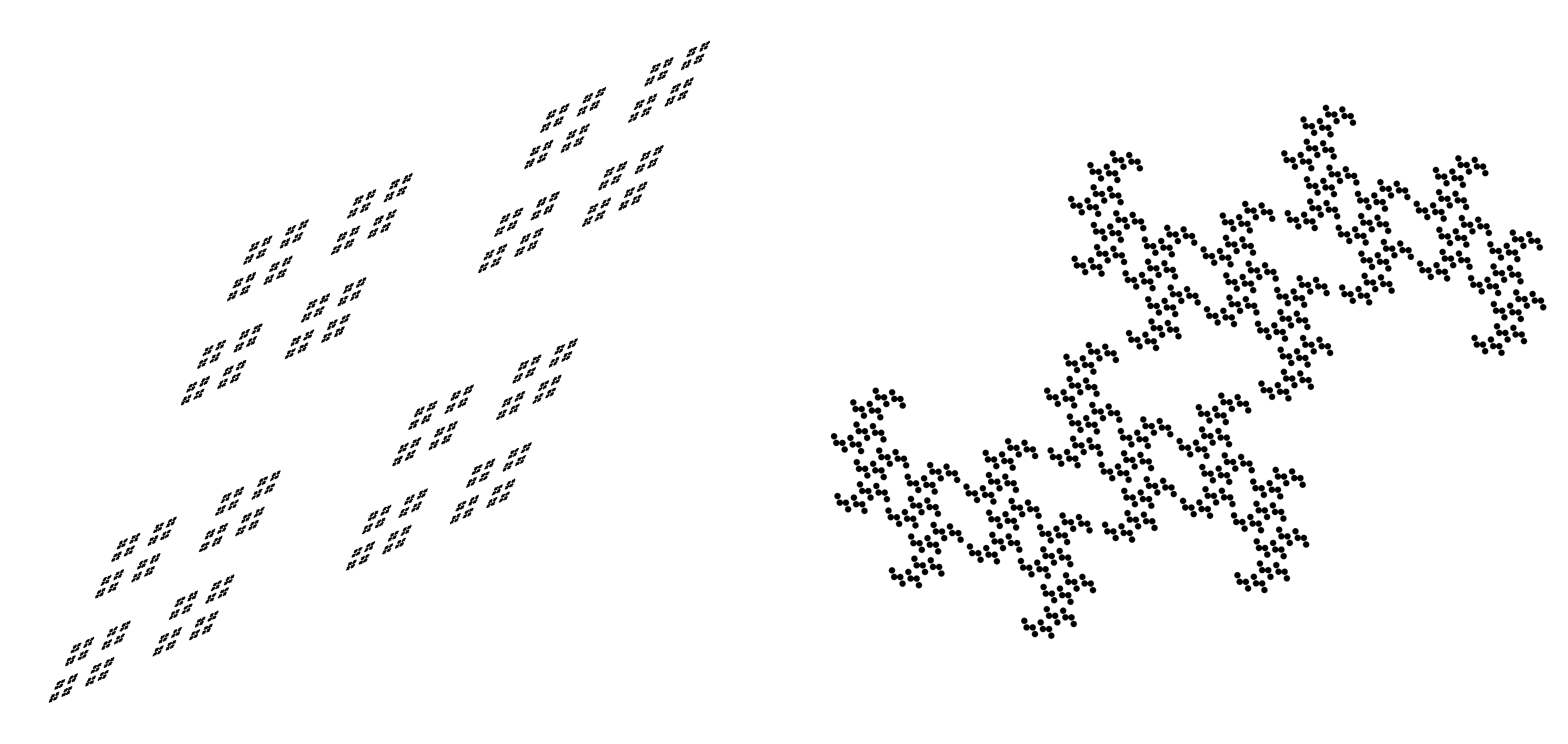}
	\caption{The self-similar set on the left has no rotations, the contraction ratio is $0.35$, its dimension is $\approx 1.32$, and it can be checked that all its projections are intervals. By Proposition \ref{prop: homo-small}, this set is tube-null, and even has tube dimension $<1$. The self-similar set on the right has irrational rotations and dimension $>1$; by Proposition \ref{prop:Pablo}, it has tube dimension $1$, but we do not know if it is tube-null.}
	\label{sss}
\end{figure}

\subsection{Bounding the number of directions in Theorem \ref{thm-principal-1}}\label{sec:R_0}
In the proof of Theorem \ref{thm-principal-1}, it was shown that $K$ may be covered by finitely many sets, each of them projecting onto a set of Hausdorff dimension $<c<1$ in some direction. We now briefly discuss how to derive quantitative estimates for the number of sets required in this covering, that is, the number $R_0$ in Lemma \ref{proj sing corollary}.
%Using Fourier analysis, we derive an upper bound that is polynomial in $R_0$. Most likely, this bound can be improved.

Let $\psi:\R^d \to \R$ be a function with absolutely converging Fourier series supported on $[0,1]^d \setminus K$, such that $\int \psi \,d\mathcal{L} = 1$. Using Parseval's equality, we see that for any probability measure $\nu$ supported on $K$,
\begin{align*}
0 = \int \psi(x) \,d\nu(x) = \sum_{n \in \Z^d} \widehat{\psi}(n)\overline{ \widehat{\nu}(n)}\,.
\end{align*}
Separating $\widehat{\psi}(0) = 1$ from the sum, we obtain the inequality
$$
1 \leq \sum_{n \neq 0} |\widehat{\psi}(n)| |\widehat{\nu}(n)|.
$$
From this, one can deduce that $R_0$ must be such that $\sum_{|n| \geq R_0} |\widehat{\psi}(n)| \geq 1$, using the absolute convergence of the Fourier series of $\psi$ and the fact that $|\hat{\nu}(n)| \leq 1$.

For example, replacing $\psi$ with a normalized product of one-dimensional tent functions supported on a cube in $[0,1]^d \setminus K$ of side-length $1/N$ and corners in $\lbrace 0, 1/N, \ldots, 1 \rbrace^d$, through a straightforward calculation of Fourier coefficients one obtains the bound
\begin{equation}\label{eq:R_0}
R_0 \leq 2^{d+1} N^{2d}.
\end{equation}
Using a compactly supported smooth bump function, the bound can be improved to
\begin{equation}\label{eq:R_1}
R_0\leq C_{\varepsilon,d} N^{1+\varepsilon}
\end{equation}
for any $\varepsilon>0$.

Notice that these bounds are valid for self-similar $T_N$-invariant sets. A general closed $T_N$-invariant set $L$ is a subset of some $T_N^m$-invariant self-similar set, where
\[m=m(L)=\min\left\{k,L\cap N^{-k}([0,1]^d+i)=\varnothing\text{ for some }i\in\{0,1,\ldots,N^k -1\}^d\right\}\,,\]
and so one has to replace $N$ by $N^m$ in the bounds \eqref{eq:R_0}--\eqref{eq:R_1}.

\subsection{Dimension drop in non-principal directions}
\label{ss:dim-drop-non-principal}

We saw in Proposition \ref{prop:dimension-drop} that for any $T_N$-invariant measure $\mu$ there is a rational direction $v$ such that $\dim P_v\mu<1$. We can also characterize those $T_N$-invariant measures for which such a dimension drop occurs in a non-principal direction:

\begin{lemma} \label{lem:dim-drop-non-principal}
Let $\mu$ be a $T_N$-invariant measure on $[0,1]^2$. Then the following are equivalent:
\begin{enumerate}
  \item The measure $\mu$ is a convex combination of measures of the form $\mu_1\times \lambda$ and $\lambda \times \mu_2$, where $\lambda$ denotes Lebesgue measure on $[0,1]$, and $\mu_i$ are $T_N$-invariant measures on $[0,1]$.
  \item $\dim P_v\mu=1$ for all $v=(v_1,v_2)\in\R^2$ such that $v_1 v_2\neq 0$.
\end{enumerate}
\end{lemma}
\begin{proof}
If $\mu=\mu_1\times\lambda$ and $v$ is a non-principal vector, then $P_v\mu$ is the convolution of scaled copies of $\mu_1$ and $\lambda$, in particular it is absolutely continuous and thus $\dim P_v\mu=1$. Likewise for $\lambda\times \mu_2$ and convex combinations of such measures.

Now let $\mu$ be a $T_N$-invariant measure such that $\dim P_v\mu=1$ for all non-principal directions $v$. We may assume $\mu$ is ergodic, for otherwise we can apply this case to the ergodic decomposition. The proofs of Lemma \ref{proj sing corollary} and Proposition \ref{prop:dimension-drop} show that if $\widehat{\mu}(v)\neq 0$, then $\dim P_v\mu<1$. Hence, all non-zero Fourier coefficients of $\mu$ must lie on some coordinate axis. If we denote the projection to the $j$-th coordinate by $\widetilde{P}_j$, this implies that
\begin{equation} \label{eq:mu-decomposition}
\mu = \widetilde{P}_1\mu \times \lambda + \lambda\times \widetilde{P}_2\mu -\lambda\times\lambda.
\end{equation}
Indeed, using that all Fourier coefficients for frequencies outside of the axes are zero, it is easy to check that the measure on the right-hand side has the same Fourier coefficients as $\mu$. Suppose neither $\widetilde{P}_1\mu$ nor $\widetilde{P}_2\mu$ are Lebesgue measure. The measures $\widetilde{P}_j\mu$ are $T_N$-invariant and ergodic on $[0,1]$ (since $\mu$ is ergodic and $\widetilde{P}_j$ is a factor map). Since they are not equal to $\lambda$, they must be mutually singular to it. Hence we can find Borel sets $A_1, A_2$ such that $\widetilde{P}_j\mu(A_j)=0$ and $\lambda(A_j)=1$ for $j=1,2$. Using \eqref{eq:mu-decomposition}, this implies the absurd fact that $\mu(A_1\times A_2)=-1$. Hence either $\widetilde{P}_1\mu$ or $\widetilde{P}_2\mu$ must be Lebesgue measure, and in light of \eqref{eq:mu-decomposition} this completes the proof.
\end{proof}

%Higher dimensions?

Let $\dimh\mu$ denote the (lower) Hausdorff dimension of a measure $\mu$; recall that it is defined as
\[
\dimh\mu = \inf\{ \dimh A: \mu(A)>0\}.
\]
Let $\mu$ be a $T_N$-invariant and ergodic measure. A key hypothesis in a recent joint equidistribution result of Algom \cite[Theorem 1.1]{Algom} is that $\dimh P_v\mu<\dimh\mu$ for some non-principal direction $v$. As a direct consequence of Lemma \ref{lem:dim-drop-non-principal}, we can characterize such measures when $\dimh\mu=1$:
\begin{corollary}
Let $\mu$ be $T_N$-invariant and ergodic on $[0,1]^2$ with $\dimh\mu=1$. Then the following are equivalent:
\begin{enumerate}
  \item $\mu$ is equal to either $\nu\times\lambda$ or $\lambda\times \nu$ for some $T_N$-invariant and ergodic measure $\nu$ of zero Hausdorff dimension on $[0,1]$.
  \item $\dimh P_v\mu=\dimh \mu=1$ for all non-principal directions $v$.
\end{enumerate}
\end{corollary}
\begin{proof}
This is immediate from Lemma \ref{lem:dim-drop-non-principal} and the following well known properties: (a) $\dimh(\nu)\le\dim(\nu)$ for all measures $\nu$, with equality for $T_N$-invariant and ergodic measures, (b) $\dim(\nu\times\lambda)=\dim(\nu)+1$.
\end{proof}

\subsection{Isotropic doubling measures}
As our last remark, we provide an application of Lemma \ref{proj sing corollary} to a geometric analysis problem. A measure $\mu$ on $\R^d$ is called \textit{isotropic doubling} if there is a constant $C<\infty$ such that for all pairs of congruent rectangles $Q_1$ and $Q_2$, with $Q_1\cap Q_2\neq\varnothing$, it holds that
\begin{equation}\label{doubling_condition}
\frac1C \leq \frac{\mu(Q_1)}{\mu(Q_2)} \leq C\,.
\end{equation}
This notion was defined by Kovalev, Maldonado, and Wu in  \cite{KMW} in connection to quasiconformal mappings.
Note that in order for \eqref{doubling_condition} to make sense, the measure $\mu$ must be fully supported. However, for the following discussion, it does not affect the generality if we restrict  $\mu$ to $[0,1]^d$ and consider only rectangles $Q_1,Q_2\subset[0,1]^d$.
In \cite{KMW} the authors propose the following open problem:
\[\text{Are there isotropic doubling measures in }\R^d\text{ with Hausdorff dimension }<d\,?\]
As shown in \cite{KMW}, isotropic doubling measures can be singular with respect to the Lebesgue measure. On the other hand, it is easy to see that on $\R^d$ they must have Hausdorff dimension at least $d-1$. In fact, their orthogonal projections onto hyperplanes must be absolutely continuous with respect to the $(d-1)$-dimensional Lebesgue measure.  Combining with Lemma \ref{proj sing corollary} this leads to the following

\begin{proposition}
For $d\ge 2$, the Lebesgue measure is the only non-zero $T_N$-invariant isotropic doubling measure on $[0,1]^d$.
\end{proposition}

\begin{proof}
As noted above, the orthogonal projection of an isotropic doubling measure on $[0,1]^d$ onto any hyperplane is absolutely continuous. Thus the proposition follows from Lemma \ref{proj sing corollary}: if the projection onto a line is not absolutely continuous, the same is true for the projection onto a hyperplane containing the line.
\end{proof}

We note that the condition \eqref{doubling_condition} being true for all rectangular shapes is essential here. In particular, there is no control over the ratios between the side-lengths. If we relax the definition and  require \eqref{doubling_condition} only for rectangles where the aspect ratio is bounded, then we arrive at the definition of classical \emph{doubling measures} (see e.g. \cite{KMW,KRS}). For $N>2$ there are $T_N$-invariant doubling measures with any dimension $0<s\le d$. For instance, we may consider the (trivial) IFS $\cF$ defined as in \eqref{eq-Nadic-IFS}, with $\#\Gamma = N^d$, which leaves $[0,1]^d$ invariant. Define $\mu$ as the self-similar measure associated to the positive probabilities $(p_\lambda)_{\lambda \in \Lambda}$, $\sum_{\lambda \in \Lambda} p_\lambda = 1$ , chosen so that  $p_\lambda$ are equal whenever $f_\lambda([0,1]^d)$ intersects the boundary of $[0,1]^d$. Each such $\mu$ is $T_N$-invariant and doubling, and depending on $(p_\lambda)$, the dimension can take any value in $]0,d]$.

%\bibliography{mybib}{}
%\bibliographystyle{plain}

\end{document}